\documentclass[a4paper,leqno]{amsart}

\usepackage{amssymb,amsthm,amsmath,amsrefs,enumerate}

\usepackage[T1]{fontenc}        % accents cod\'{e}s dans la fonte

\usepackage[latin1]{inputenc}   % accents 8 bits dans le fichier

\newcommand{\la}{\lambda}

\newcounter{mylisti} \newcounter{mylistii}
\newcounter{nest}
\newcommand{\defaultlabel}{}

\newcommand{\bn}{\ensuremath{\mathbb N}}

\newcommand{\br}{\ensuremath{\mathbb R}}

\newcommand{\bz}{\ensuremath{\mathbb Z}}

\newcommand{\cC}{\ensuremath{\mathcal C}}

\newcommand{\cP}{\ensuremath{\mathcal P}}

\newcommand{\cU}{\ensuremath{\mathcal U}}
\newcommand{\cV}{\ensuremath{\mathcal V}}

\newcommand{\diam}{\operatorname{diam}}

\newcommand{\ds}{\displaystyle}

\newtheorem{thm}{Theorem}

\newtheorem{lem}{Lemma}
\newtheorem{claim}{Claim}
\newtheorem{prop}{Proposition}
\newtheorem{cor}{Corollary}
\newtheorem{problem}{Problem}

\theoremstyle{definition}

\newtheorem{defn}{Definition}

\theoremstyle{remark}

\newcommand{\eps}{\varepsilon}

\newcommand{\xn}{\ensuremath{(x_n)_{n=1}^\infty}}
\newcommand{\yn}{\ensuremath{(y_n)_{n=1}^\infty}}

\newcommand{\xbgq}{\ensuremath{\left(\sum_{n=1}^\infty\ell_{p_n}\right)_{\ell_q}}}

\begin{document}

\title[Tight embeddability of proper and stable metric spaces]{Tight embeddability of proper and stable metric spaces}

\author[F.~Baudier]{F.~Baudier$^\dag$}
\address{$\dag$ Institut de Math\'ematiques Jussieu-Paris Rive Gauche, Universit\'e Pierre et Marie Curie, Paris, France and Department of Mathematics, Texas A\&M University, College Station, TX 77843, USA (current)}
\email{flo.baudier@imj-prg.fr, florent@math.tamu.edu (current)}
\author[G.~Lancien]{G.~Lancien$^\ddag$}
\address{$\ddag$ Laboratoire de Math\'ematiques de Besan\c con, Universit\'e de Franche-Comt\'e, 16 route de Gray, 25030 Besan\c con C\'edex, France}
\email{gilles.lancien@univ-fcomte.fr}
\date{}

\thanks{The first author's research was partially supported by ANR-13-PDOC-0031, project
  NoLiGeA}
\keywords{}
\subjclass[2010]{46B85, 46B20}

\begin{abstract} We introduce the notions of almost Lipschitz embeddability and nearly isometric embeddability. We prove that for $p\in [1,\infty]$, every proper subset of $L_p$ is almost Lipschitzly embeddable into a Banach space $X$ if and only if $X$ contains uniformly the $\ell_p^n$'s. We also sharpen a result of N. Kalton by showing that every stable metric space is nearly isometrically embeddable in the class of reflexive Banach spaces.
\end{abstract}

\maketitle

\section{Introduction}
Let $(X,d_X)$ and $(Y,d_Y)$ be two metric spaces and $f$ a map from $X$ into $Y$. We define the compression modulus of $f$ by $\rho_f(t):=\inf\{d_Y(f(x),f(y))\colon d_X(x,y)\geq t\},$ and the expansion modulus by $\omega_f(t):=\sup\{d_Y(f(x),f(y))\colon d_X(x,y)\leq t\}.$ One can define various notions of embeddability by requiring the compression and the expansion moduli to satisfy certain properties. For instance, $X$ is said to be bi-Lipschitzly embeddable into $Y$ if there exists a map from $X$ into $Y$ with a sublinear expansion modulus and a superlinear compression modulus. The notion of bi-Lipschitz embeddability is arguably one of the most studied and important such notion, due to its numerous and fundamental applications in computer science, as well as its ubiquity in theoretical mathematics. Despite being much more flexible than isometric embeddability, the notion of bi-Lipschitz embeddability is still rather restrictive. Certain natural relaxations have been considered and heavily studied over the years (e.g. uniform embeddability in geometric Banach space theory \cite{BenyaminiLindenstrauss2000}). A few of them turned out to be extremely useful in various scientific fields such as: coarse embeddability in topology \cite{NowakYu2012}, quasi-isometric embeddability in geometric group theory \cite{Gromov1993}, or range embeddability in computer science (e.g. for proximity problems including nearest neighbor search or clustering, cf. \cite{OstrovskyRabani2002} and \cite{BartalGottlieb2014}).

In this article we introduce two new notions of metric embeddability that are slightly weaker than bi-Lipschitz embeddability but much stronger than strong embeddability. Recall that $X$ is coarsely embeddable (resp. uniformly embeddable) into $Y$ if there exists $f:X\to Y$ such that $\lim_{t\to \infty}\rho_f(t)=\infty$ and $\omega_f(t)<\infty$ for every $t\in(0,\infty)$ (resp. $\lim_{t\to 0}\omega_f(t)=0$ and $\rho_f(t)>0$ for every $t\in(0,\infty)$). $X$ is strongly embeddable into $Y$ if it is simultaneously coarsely and uniformly embeddable, i.e. there exists $f:X\to Y$ such that $\lim_{t\to \infty}\rho_f(t)=\infty$, $\lim_{t\to 0}\omega_f(t)=0$, and $(\rho_f(t),\omega_f(t))\in(0,\infty)^2$ for every $t\in(0,\infty)$. As we will see our new notions also imply range embeddability in the sense of Bartal and Gottlieb \cite{BartalGottlieb2014}. We also want to point out that our approach and the way we relax the notion of bi-Lipschitz embeddability differs from the previously mentioned above. In all the notions above two classes of (real valued) functions with a certain behavior are fixed (one for each of the two moduli), and one asks for the existence of a \textit{single} embedding whose compression and expansion moduli belong to the related classes. In our approach we still fix two classes of functions with some prescribed properties, but we require the existence of a \textit{collection} of embeddings whose compression and expansion moduli belong to the related classes. In Section \ref{propersection} we introduce the notion of almost Lipschitz embedabbility, and prove that every proper subset of $L_p$ with $p\in[1,\infty]$ is almost Lipschitzly embeddable into any Banach space containing uniformly the $\ell_p^n$'s. In Section \ref{stablesection} we introduce the notion of nearly isometric embeddability. We then show, using a slight modification of a result of Kalton in \cite{Kalton2007}, that any stable metric space is nearly isometrically embeddable into the class of reflexive Banach spaces. We discuss the optimality of our results in Section \ref{optimality}. Finally in the last section we discuss the relationship between the various notions of embeddability and raise a few questions.

\section{Almost Lipschitz embeddability of proper metric spaces}\label{propersection}
A metric space is said to be proper if all its closed balls are compact. The class of proper metric spaces is rather large since it obviously contains all locally finite or compact metric spaces, but also all finite-dimensional Banach spaces and all compactly generated groups. The main result of \cite{Baudier2012} states that every proper metric space is strongly embeddable into any Banach space with trivial cotype. More precisely, an embedding $f$ that is Lipschitz and such that $\rho_f$ behaves essentially like $t\mapsto\frac{t}{\log^2(t)}$ (up to some universal constants), is constructed. This embedding is ``almost'' bi-Lipschitz, up to the logarithmic factor. An equivalent reformulation of the main result of \cite{Baudier2012} will say that every proper subset of $L_\infty$ is strongly embeddable into any Banach space containing the $\ell_\infty^n$'s uniformly. Two problems naturally arise. Is it possible to improve the quality of the embedding, and get even closer to the properties of a bi-Lipschitz one? Does an analogue result hold where $L_\infty$ and $\ell_\infty^n$'s are replaced by $L_p$ and $\ell_p^n$'s, respectively? The construction in \cite{Baudier2012} combines a gluing technique introduced in \cite{Baudier2007}, that is tailored for locally finite metric spaces (as shown in \cite{BaudierLancien2008}, and ultimately in \cite{Ostrovskii2012}), together with a net argument suited for proper spaces, whose inspiration goes back to \cite{BenyaminiLindenstrauss2000}. Due to the technicality of this construction, it is difficult and unclear how to make any progress towards the two problems. Indeed, the gluing technique requires a slicing of the proper space with annuli of increasing widths, while the net argument involves nets with certain separation parameters, but the widths and the separation parameters were correlated in \cite{Baudier2012}. In this section we answer affirmatively the two questions above. We first ``decorrelate'' the ingredients of the construction from \cite{Baudier2012}, in the sense that the widths of the annuli and the separation parameters can be chosen independently from each other. Doing so we have much more flexibility to construct far better behaved embeddings. Finally the collection of local embeddings that will ultimately be pasted together to produce the global embedding desired, are obtained using the bounded approximation property instead of Fr\'echet type embeddings, which allow us to treat all the $L_p$ spaces at once. Eventually we obtain a construction that provides tight embeddings in much greater generality, and that is significantly simpler. We start by introducing the notion of \textit{almost Lipschitz embeddability}.

\begin{defn}\label{almostLip} Let $\Phi:=\{\varphi\colon[0,\infty)\to [0,1)\ |\ \varphi \textrm{ is continuous, } \varphi(0)=0,\textrm{ and } \varphi(t)>0 \textrm{ for all }t>0\}.$ We say that $(X,d_X)$ almost Lipschitzly embeds into $(Y,d_Y)$ if there exist a scaling factor $r\in(0,\infty)$, a constant $D\in[1,\infty)$, and a family $(f_\varphi)_{\varphi\in\Phi}$ of maps from $X$ into $Y$ such that for all $t\in(0,\infty),$ $\omega_{f_\varphi}(t)\le Drt\ \ {\rm and}\ \ \rho_{f_\varphi}(t)\ge rt\varphi(t).$
In other words, $(X,d_X)$ is almost Lipschitz embeddable into $(Y,d_Y)$ if there exist a scaling factor $r\in(0,\infty)$ and a constant $D\in[1,\infty)$, such that for any continuous function $\varphi\colon [0,+\infty)\to [0,1)$ satisfying $\varphi(0)=0$ and $\varphi(t)>0$ for all $t>0$, there exists a map $f_\varphi\colon X\to Y$ such that for all $x,y\in X$,
$$\varphi(d_X(x,y))rd_X(x,y)\le d_Y(f_\varphi(x),f_\varphi(y))\le Drd_X(x,y).$$
\end{defn}

It is clear from Definition \ref{almostLip} that if $X$ admits a bi-Lipschitz embedding into $Y$, then $X$ almost Lipschitzly embeds into $Y$. The notion of almost Lipschitz embeddability expresses the fact that one can construct an embedding that is as close as one wishes to a bi-Lipschitz embedding. Note also that if $X$ almost Lipschitz embeds into $Y$, then $X$ strongly embeds into $Y$. More precisely for every $\varphi\in\Phi$, $X$ admits a strong embedding into $Y$, whose quality depends on the properties of the function $\varphi$.

In the sequel it will be interesting to notice that to achieve almost Lipschitz embeddability, it is sufficient to consider only exponential-type functions. We record this fact in Lemma \ref{reduction}.

\begin{lem}\label{reduction}
For any continuous function $\varphi\colon [0,+\infty)\to [0,1)$ such that $\varphi(0)=0$ and $\varphi(t)>0$ for all $t>0$, there exists a continuous non-decreasing surjective function $\mu\colon (0,+\infty)\to (-\infty,0)$, so that $\varphi(t)\le2^{\mu(t)}$ for all $t>0$.
\end{lem}
\begin{proof} By replacing $\varphi$ by $t\mapsto\max\{\varphi(t),1-e^{-t}\}$, we may assume that $\varphi \colon [0,+\infty)\to [0,1)$ is onto. Then, let
$$\begin{array}{rcl}
     \lambda\ :\ (0,+\infty)  & \rightarrow & (0,1)\\
       &   &  \\
     t & \mapsto & \sup\{\varphi(s)\colon s\in(0,t]\},\\
    \end{array}$$

and define
$$\begin{array}{rcl}
     \mu\ :\ (0,+\infty)  & \rightarrow & (-\infty, 0)\\
       &   &  \\
     t & \mapsto &\log_2\lambda(t).
    \end{array}$$

It is easy to verify that:
\begin{enumerate}[(a)]
\item $\varphi(t)\le\lambda(t)=2^{\mu(t)}$, for all $t>0$.
\item $\lambda$ is continuous, and hence $\mu$ is continuous as well.
\item $\mu$ is non-decreasing and surjective onto $(-\infty,0)$.
\end{enumerate}
\end{proof}
It is worth noticing that almost Lipschitz embeddability implies the existence of range embeddings with arbitrarily large unbounded ranges. Let $I$ be an interval in $[0,\infty)$. Following \cite{BartalGottlieb2014} we say that $X$ admits a bi-Lipschitz range embedding with range $I$, or simply an $I$-range embedding, into $Y$ if there exist $f\colon X\to Y$, a scaling factor $r\in(0,\infty)$, and a constant $D\in[1,\infty)$, such that $rd_X(x,y)\le d_Y(f(x),f(y))\le Drd_X(x,y)$, whenever $d_X(x,y)\in I$. It is easily checkable that if $X$ almost Lipschitzly embeds into $Y$, then there exist a scaling factor $r\in(0,\infty)$, and a constant $D\in[1,\infty)$, such that for every $s\in(0,\infty)$ by choosing $\varphi$ appropriately in $\Phi$, there exists a map $f_s\colon X\to Y$ so that $ rd_X(x,y)\le d_Y(f_s(x),f_s(y))\le Drd_X(x,y)$, for every $x,y\in X$ satisfying $d_X(x,y)\in[s,\infty)$. Therefore, we have:

\begin{prop}\label{range} If $X$ is almost Lipschitzly embeddable into $Y$ then for every $s\in(0,\infty)$, $X$ admits a $[s,\infty)$-range embedding into $Y$
\end{prop}
The notion of almost Lipschitz embeddability can be thought of as a quantitative refinement of strong embeddability. This quantitative notion of embeddability is much finer than simply estimating the large-scale and small-scale quantitative behaviors of the compression and expansion moduli of strong embeddings, and is of course much harder to obtain. It follows from Proposition \ref{range} that if $X$ is almost Lipschitzly embeddable into $Y$ then $X$ admits a quasi-isometric embedding into $Y$ in the sense of Gromov. Note also that the notion of almost Lipschitz embeddability does not forget any intermediate scales whereas strong embeddability might.

\medskip

Before stating the main result of this section, we need to recall some basic facts and definitions about the local theory of Banach spaces. For $\lambda\in[1,\infty)$ and $p\in[1,\infty]$, we say that a Banach space $X$ contains the $\ell_p^n$'s $\lambda$-uniformly if there is a sequence of subspaces $(E_n)_{n=1}^\infty$ of X such that $E_n$ is linearly isomorphic to $\ell_p^n$ with $d_{BM}(E_n,\ell_p^n)\le \lambda$, where $d_{BM}$ denotes the Banach-Mazur distance between Banach spaces. We say that $X$ contains the $\ell_p^n$'s uniformly if it contains the $\ell_p^n$'s $\lambda$-uniformly for some $\lambda\in[1,\infty)$. Let $X$ be a Banach space, $p\in [1,2]$ and $q\in [2,\infty]$ and let $(\eps_i)_{i=1}^\infty$ be a sequence of independent Rademacher variables on a probability space $\Omega$.\\
We say that $X$ has type $p$ if there exists $T\in(0,\infty)$ such that for all $x_1,\dots,x_n \in X$

$$\big\|\sum_{i=1}^n \eps_ix_i\big\|_{L_2(\Omega,X)}\le T (\sum_{i=1}^n \|x_i\|^p)^{1/p},$$
and $X$ has cotype $q$ if there exists $C\in(0,\infty)$ such that for all $x_1,\dots,x_n \in X$

$$\left(\sum_{i=1}^n \|x_i\|^q\right)^{1/q}\le C\big\|\sum_{i=1}^n \eps_ix_i\big\|_{L_2(\Omega,X)}.$$

The following statement gathers the deep links between these notions that have been proved in the works of James \cite{James1964}, Krivine \cite{Krivine1976}, Maurey and Pisier \cite{MaureyPisier1976}.

\begin{thm}\label{JKMP} Let $X$ be an infinite-dimensional Banach space. If we denote
$$p(X)=\sup\{p\ \colon\ X\ {\rm is\ of\ type}\ p\}\ \ {\rm and}\ \ q(X)=\inf\{q\ \colon \ X\  {\rm is\ of\ cotype}\ q\},$$
then $X$ contains the $\ell_{p(X)}^n$'s and the $\ell_{q(X)}^n$'s uniformly.\\
Moreover, if $X$ contains the $\ell_r^n$'s uniformly, then $r\in [p(X),q(X)]$ and $X$ contains the $\ell_r^n$'s $\lambda$-uniformly, for all $\lambda>1$.
\end{thm}

The following corollary will be useful to us. It is certainly well-known, but we include a proof for the sake of completeness.

\begin{cor}\label{finitecodim} Let $Y$ be an infinite-dimensional Banach space and $Z$ a finite co-dimensional subspace of $Y$. Assume that $Y$ contains the $\ell_p^n$'s uniformly, then $Z$ contains the $\ell_p^n$'s $\lambda$-uniformly, for all $\lambda>1$.
\end{cor}

\begin{proof} Assume first that $p=1$. We deduce that $p(Y)=1$. It is easily checked that $p(Y)=p(Z)$. Then the conclusion follows from Theorem \ref{JKMP}.\\
Assume now that $p\in (1,+\infty]$ and let $q\in [1,\infty)$ be the conjugate exponent of $p$. We can write $Y=Z\oplus F$, where $F$ is finite dimensional. There exist $f_1,\dots,f_m$ in the unit sphere of $F$ and $u_1,\dots,u_m\in Y^*$ such that $P(x)=\sum_{i=1}^m u_i(x)f_i$. Let us fix $n\in \mathbb N$. We will look for a copy of $\ell_p^n$ in $Z$. For that purpose we choose $\eps>0$ (very small and to be chosen later) and $N\in \mathbb N$ large enough (the size of $N$ will depend on $n$ and $\eps$ and also be made precise later). Then there exists a subspace $X_N$ of $Y$ and an isomorphism $T_N:\ell_p^N \mapsto X_N$ with $\|T_N\|\le 1$ and $\|T_N^{-1}\|\le 2$. We denote $(e_j)_{j=1}^N$ the image by $T_N$ of the canonical basis of $\ell_p^N$. Denote $v_i$ the restriction of $u_i$ to $X_N$. Then $v_i \circ T_N \in (\ell_p^N)^*=\ell_q^N$ and $\|v_i \circ T_N\|_q\le \|u_i\|$. Then the cardinality of $\{j\in\{1,\dots,N\}\colon |u_i(e_j)|\ge \eps\}$ is at most $\frac{\|u_i\|^q}{\eps^q}$. If $N$ was chosen large enough, then there is a subset $A$ of $\{1,\dots,N\}$ of cardinality $n$ such that for all $j\in A$ and all $i\in \{1,\dots,m\}$, $|u_i(e_j)|< \eps$. Let $X_A$ be the linear span of $\{e_j\colon j\in A\}$ and note that $X_A$ is 2-isomorphic to $\ell_p^n$. Then for all $i\in \{1,\dots,m\}$, the norm of the restriction to $X_A$ of $u_i$ is at most $n\eps$. It follows that the norm of the restriction of $P$ to $X_A$ does not exceed $nm\eps$. Therefore, if $\eps$ was chosen small enough, $I-P$ is an isomorphism from $X_A$ onto its image in $Z$, with distortion less than 2. We have shown that $Z$ contains the $\ell_p^n$'s 4-uniformly. Finally, we use Theorem \ref{JKMP} to conclude that $Z$ contains the $\ell_p^n$'s $\lambda$-uniformly, for all $\lambda>1$.
\end{proof}

We now have all the ingredients to state and prove the technical lemma on Banach spaces containing the $\ell_p^n$'s that will serve our purpose.

\begin{lem}\label{fdd} Let $\eps>0$, $\gamma >0$, $(d_j)_{j=0}^\infty$ a sequence in $\mathbb N$, and $Y$ be a Banach space containing uniformly the $\ell_p^n$'s. Then there exists a sequence $(H_j)_{j=0}^\infty$ of subspaces of $Y$ so that for any $j\ge 0$, $d_{BM}(H_j,\ell_p^{d_j})\le 1+\eps$, and also such that $(H_j)_{j=0}^\infty$ is a Schauder finite-dimensional decomposition of its closed linear span $Z$ with $\|P_j\|\le 1+\gamma$, where $P_j$ is the projection from $Z$ onto $H_0\oplus\dots\oplus H_j$ with kernel $\overline{\rm Span}\,(\bigcup_{i=j+1}^\infty H_i)$.
\end{lem}

\begin{proof} Pick a sequence $(\gamma_j)_{j=0}^\infty$ with $\gamma_j>0$ and $\Pi_{j=0}^\infty (1+\gamma_j)\le 1+\gamma$. Choose first a subspace $H_0$ of $Y$ so that $d_{BM}(H_0,\ell_p^{d_0})\le 1+\eps$. Assume that $H_0,\dots,H_j$ have been constructed. Then, using the standard Mazur technique, we can find a finite co-dimensional subspace $Z$ of $Y$ so that
$$\forall y\in H_0+\dots+H_j,\ \forall z\in Z,\ \|y\|\le (1+\gamma_j)\|y+z\|.$$
By Lemma \ref{finitecodim}, we can now pick a subspace $H_{j+1}$ of $Z$ such that  $d_{BM}(H_{j+1},\ell_p^{d_{j+1}})$. It is now classical that the sequence $(H_j)_{j=0}^\infty$ satisfies the desired properties.
\end{proof}

We can now state our main result.
\begin{thm}\label{Lp0} Let $p\in [1,+\infty]$, and $Y$ be a Banach space containing uniformly the $\ell_p^n$'s. Let $M$ be a proper subset of $L_p$. For every $r<\frac{1}{16}$ and $\mu\colon (0,+\infty)\to (-\infty,0)$ a surjective continuous non-decreasing function, there exists a map $f\colon M\to Y$ such that for every $x,y\in M$, $$2^{\mu(\|x-y\|_p)}r\|x-y\|_p\le\| f(x)-f(y)\|_Y\le 9\|x-y\|_p.$$
\end{thm}

The following corollary is now a simple consequence of Theorem \ref{Lp0}, Lemma \ref{reduction}, and Definition \ref{almostLip}.
\begin{cor}\label{Lp} Let $p\in [1,+\infty]$, $M$ be a proper subset of $L_p$, and $Y$ be a Banach space containing uniformly the $\ell_p^n$'s. Then $M$ is almost Lipschitz embeddable into $Y$.
\end{cor}

Before proving Theorem \ref{Lp0} we derive two interesting corollaries. Since $L_\infty$ is isometrically universal for all separable metric spaces. The case $p=\infty$ yields the following improvement of a result of the first named author \cite{Baudier2012}.

\begin{cor}\label{nocotype}
Let $(M,d_M)$ be a proper metric space and $Y$ be a Banach space without cotype,
then $M$ is almost Lipschitz embeddable into $Y$.
\end{cor}

We can also use Dvoretsky's theorem to deduce the following from the case $p=2$.

\begin{cor}\label{hilbert}
Let $M$ be a proper subset of a Hilbert space $H$ and $Y$ be an infinite-dimensional Banach space, then $M$ is almost Lipschitz embeddable into $Y$.
\end{cor}

The remainder of this section will be devoted to the proof of Theorem \ref{Lp0}. Recall that a Banach space $X$ has the $\lambda$-bounded approximation property ($\lambda$-BAP in short), if for any $\varepsilon\in(0,\infty)$, and any compact subset $K\subset X$, there exists a finite-dimensional subspace $G$ of $X$ and a linear map $\varphi \colon X \to G$ such that $\varphi$ is $\lambda$-Lipschitz, and $\|\varphi(x)-x\|_X\le \varepsilon$ for all $x\in K$.

\begin{proof}[Proof of Theorem \ref{Lp0}]

Let $\sigma$ denote the generalized inverse of the map $\mu$, i.e. $\sigma(y)=\inf\{x\in (0,\infty)\colon \mu(x)\ge y\}$ with the convention that $\inf\emptyset=\infty$. The following properties of $\sigma$ are crucial in the sequel (we refer to \cite{EmbrechtsHofert2013} for more information on generalized inverses):
\begin{enumerate}[(a)]
\item $\sigma\colon (-\infty,0)\to (0,+\infty)$ is non-decreasing and $\ds\lim_{t\to-\infty}\sigma(t)=0$,
\item $\mu(s)\ge t$ implies that $s\ge\sigma(t)$,
\item since $\mu$ is continuous, $\sigma(t)\le s$ implies that $t\le \mu(s)$
\end{enumerate}
In particular, it follows from $(b)$ and $(c)$ that for all $t<t'$, $\sigma(t)\le s<\sigma(t')$ is equivalent to $t\le  \mu(s)<t'$.

Fix now some parameters $\beta,\varepsilon,\gamma>0$, and $\eta>2$ to be chosen later. Since $M$ is a proper subset of $L_p$, for every $k\in\bz$ the set $B_k:=\{x\in M\ \colon\ \|x\|_p\le 2^{k+1}\}$ is a compact of $L_p$. It is well-known (cf. \cite{Handbookvol1} Chapter 7) that $L_p$ has the metric approximation property, i.e. the $1$-BAP. Therefore, for any $n\in\bn$, there exists a finite-dimensional subspace $G_{k,n}$ of $L_p$ and a linear map $\varphi_k^n\colon L_p \to G_{k,n}$ such that $\varphi_k^n$ is $1$-Lipschitz and
\begin{equation}\label{key}
\forall x\in B_k,\ \ \|\varphi_k^n(x)-x\|_p\le \frac{\sigma(-n)}{\eta}.
\end{equation}

Since $L_p$ is a $\mathcal L_{(1+\beta),p}$-space, there exists a linear embedding $R_{k,n}$ from $G_{k,n}$ onto  a subspace of some $\ell_p^{d(k,n)}$ (where $d(k,n)\in\bn$), and such that $\|R_{k,n}\|\le 1$ and $\|R_{k,n}^{-1}\|\le1+\beta$. We shall now use the fact that $Y$ contains uniformly the $\ell_p^n$'s. So, let $\psi\colon\bz\times\bn\to\bn$ be a bijection. It now follows from Lemma \ref{fdd} that we can build finite-dimensional subspaces
$(H_j)_{j=0}^\infty$ of $Y$ and $(S_j)_{j=0}^\infty$ so that, for every $j\geq
0$, $S_j$ is a linear map from $\ell_p^{d(\psi^{-1}(j))}$ onto $H_j$ satisfying
$$\forall x\in \ell_p^{d(\psi^{-1}(j))},\ \ \ \frac{1}{1+\varepsilon}\|x\|\leq \|S_jx\|_Y\leq \|x\|,$$ and also such that $(H_j)_{j=0}^\infty$ is a Schauder finite-dimensional decomposition of its closed linear span $H$. Let $P_j$ be the projection from $H$ onto $H_0\oplus...\oplus H_j$ with kernel $\overline{\rm
Span}\,(\bigcup_{i=j+1}^\infty H_i)$. Lemma \ref{fdd} also insures that this construction can be done so that $\|P_j\|\leq (1+\gamma)$, for all $j\ge 0$. We denote $\Pi_0=P_0$ and $\Pi_j=P_j-P_{j-1}$ for $j\geq 1$. Clearly $\|\Pi_j\|\leq 2(1+\gamma)$. Finally, for $(k,n)\in \bz\times\bn$, let $T_{k,n}:=S_{\psi(k,n)}\circ R_{k,n}\colon G_{k,n}\to H_{\psi(k,n)}$, which is an isomorphism from $G_{k,n}$ onto a subspace $F_{\psi(k,n)}$ of $H_{\psi(k,n)}$. Note that  $(F_j)_{j=0}^\infty$ is a Schauder decomposition of its closed linear span $Z$. Denote now  $Q_k^n=T_{k,n}^{-1}\circ\widehat{\Pi}_{\psi(k,n)}\colon Z\to G_{k,n}$, where $\widehat{\Pi}_{j}$ is the restriction of $\Pi_j$ to $Z$. Then set $f_k^n=T_{k,n}\circ \varphi_k^n$ which is defined on $L_p$, and in particular on $B_k$ and takes values in $Y$, more precisely in $F_{\psi(k,n)}= \widehat{\Pi}_{\psi(k,n)}(Z)$. Then we have  that for all $x,y \in B_k$,
$$\frac{1}{(1+\varepsilon)(1+\beta)}\big(\|x-y\|_p-\frac{2\sigma(-n)}{\eta}\big)\le \|f_k^n(x)-f_k^n(y)\|_Y\le \|x-y\|_p.$$
Define
$$\begin{array}{rcl}
     f_k\ :\ B_k  & \rightarrow &\displaystyle \sum_{n\in\bn} \widehat{\Pi}_{\psi(k,n)}(Z)\\
       &   &  \\
     x & \mapsto &\displaystyle \sum_{n=1}^\infty2^{-n} f_k^n(x).\\
    \end{array}$$
It is clear that $f_k$ is $1$-Lipschitz.\\
Define
$$\begin{array}{rcl}

    f\ :\ M & \rightarrow & Y \\
            &   &  \\
          x & \mapsto & \la_x f_{k}(x)+(1-\la_x)f_{k+1}(x)\ ,\ {\rm if}\ 2^{k}\le\|x\|_p\le 2^{k+1}\textrm{ for some }k\in\bz,
    \end{array}$$
    where $$\la_x:=\frac{2^{k+1}-\|x\|_p}{2^{k}}.$$

The rest of the proof will be divided into two parts. In the first part it is proven that $f$ is a Lipschitz map, and the proof only requires the fact that $\varphi_k^n$, $f_k^n$ and $f_k$ are Lipschitz maps, the definition of $f$ in terms of the dyadic slicing, and the triangle inequality. In the second part an estimate of the compression function of $f$ that relies essentially on the inequality \eqref{key} is given.

\medskip

\textbf{Part I: $f$ is a Lipschitz map}

\medskip

\noindent  Let $x,y \in M$ and assume, as we may, that $\|x\|_p\leq \|y\|_p$. Various cases have to be considered. Note that $f(0)=0$ since the $f_k^n$'s are linear.

\noindent \underline{Case 1.}  If $\|x\|_p\leq \frac{1}{2}\|y\|_p$, then
$$\|f(x)-f(y)\|_Y\leq \|x\|_p+\|y\|_p\leq \frac{3}{2}\|y\|_p\leq 3(\|y\|_p-\|x\|_p)\leq
3\|x-y\|_p.$$

\noindent \underline{Case 2.} If $\frac{1}{2}\|y\|_p<\|x\|_p\leq \|y\|_p$, consider two
subcases.

\smallskip

\noindent \underline{Case 2.a.} $2^k\leq \|x\|_p\leq \|y\|_p<2^{k+1}$, for some $k\in\bz$. Then,
let
$$\la_x=\frac{2^{k+1}-\|x\|_p}{2^{k}}\ \ {\rm and}\ \ \la_y=\frac{2^{k+1}-\|y\|_p}{2^{k}}.$$
\indent We have,
$$|\la_x-\la_y|=\frac{\|y\|_p-\|x\|_p}{2^{k}}\leq \frac{\|x-y\|_p}{2^{k}}$$
\indent therefore,

\begin{align*}
\|f(x)-f(y)\|_Y=  & \|\la_x f_k(x)-\la_y f_k(y)+(1-\la_x)f_{k+1}(x)-(1-\la_y)f_{k+1}(y)\|_Y
\end{align*}
\begin{align*}
                     \le & \la_x\| f_k(x)-f_k(y)\|_Y +(1-\la_x)\| f_{k+1}(x)-f_{k+1}(y)\|_Y+2|\la_x-\la_y|\|y\|_p\\
                      \le & \|x-y\|_p+2^{k+2}|\la_x-\la_y|\\
                      \le & 5\|x-y\|_p.
\end{align*}

\noindent \underline{Case 2.b.} $2^k\leq \|x\|_p<2^{k+1}\leq \|y\|_p<2^{k+2}$, for some $k\in\bz$. Then,
let
$$\la_x=\frac{2^{k+1}-\|x\|_p}{2^{k}}\ \ {\rm and}\ \ \la_y=\frac{2^{k+2}-\|y\|_p}{2^{k+1}}.$$
\indent We have,
$$\la_x \leq \frac{\|x-y\|_p}{2^k},\ \ {\rm so}\ \ \la_x\|x\|_p\leq 2\,\|x-y\|_p.$$
\indent Similarily,
$$1-\la_y=\frac{\|y\|_p-2^{k+1}}{2^{k+1}}\leq \frac{\|x-y\|_p}{2^{k+1}}\ \ {\rm
and}\ \ (1-\la_y)\|y\|_p\leq 2\,\|x-y\|_p.$$

It follows that,
\begin{align*}
\|f(x)-f(y)\|_Y&=\|\la_x f_k(x)+(1-\la_x)f_{k+1}(x)-\la_y f_{k+1}(y)-(1-\la_y)f_{k+2}(y)\|_Y
\end{align*}
\begin{align*}
              &\le \la_x(\|f_k(x)\|_Y+\|f_{k+1}(x)\|_Y)+(1-\la_y)(\|f_{k+1}(y)\|_Y+\|f_{k+2}(y)\|_Y)\\
              & \hskip 8cm+\|f_{k+1}(x)-f_{k+1}(y)\|_Y\\
              &\le2\la_x\|x\|_p+2(1-\la_y)\|y\|_p+\|x-y\|_p\\
              &\le9\|x-y\|_p.
\end{align*}

\smallskip\noindent We have shown that $f$ is $9$-Lipschitz.

\smallskip

\textbf{Part II: Estimation of the compression modulus of $f$.}

\medskip

\noindent Estimating from below the compression modulus of $f$ requires the investigation of three different cases, based on the location of the pair of points in $M$ that we are considering. The three situations are as follows: the two points are in the same dyadic annulus, or in two consecutive dyadic annuli, or eventually separated by at least an entire dyadic annulus. Let $x,y\in M$. In the following discussion we will assume that $\|x\|_p\leq \|y\|_p$, $2^k\le \|x\|_p<2^{k+1}$, and $2^{\ell}\le \|y\|_p<2^{\ell+1}$, for some $k,\ell\in\bz$.
Recall that $Q_k^n=T_{k,n}^{-1}\circ\widehat{\Pi}_{\psi(k,n)}$, and hence $\|Q_k^n\|\le\|T_{k,n}^{-1}\|\times\|\widehat{\Pi}_{\psi(k,n)}\|
\le\|R_{k,n}^{-1}\|\times\|S_{\psi(k,n)}^{-1}\|\times\|\Pi_{\psi(k,n)}\|\le (1+\beta)\times(1+\varepsilon)\times 2(1+\gamma)$.

\medskip
Assume first that $\|x-y\|_p<\sigma(-1)$ and pick $n\in\bn$ such that $\sigma(-(n+1))\le \|x-y\|_p<\sigma(-n)$ or equivalently $-(n+1)\le \mu(\|x-y\|_p)< -n$. Then
$$Q_{k}^{n+1}(f(x))=2^{-(n+1)}\la_x\varphi_k^{n+1}(x),\ Q_{k+1}^{n+1}(f(x))=2^{-(n+1)}(1-\la_x)\varphi_{k+1}^{n+1}(x),$$
$$Q_{\ell}^{n+1}(f(x))=2^{-(n+1)}\la_x\varphi_{\ell}^{n+1}(x),\ Q_{\ell+1}^{n+1}(f(x))=2^{-(n+1)}(1-\la_x)\varphi_{\ell+1}^{n+1}(x),$$

and
$$Q_r^{n+1}(f(x))=Q_s^{n+1}(f(y))=0\ \textrm{for}\ r\notin\{k,k+1\},\ s\notin\{\ell,\ell+1\}.$$
So in each of the three cases described in our foreword for the proof of Part II, we can write
\begin{align*}
(Q_{r_1}^{n+1}+\cdots+Q_{r_s}^{n+1})(f(x)-f(y))=&2^{-(n+1)}[\la_x\varphi_k^{n+1}(x)+(1-\la_x)\varphi_{k+1}^{n+1}(x)\\
                                                              &-\la_y\varphi_{\ell}^{n+1}(y)-(1-\la_y)\varphi_{\ell+1}^{n+1}(y)]
\end{align*}
with $s\in \{2,3,4\}$ and $r_1,\dots,r_s \in \{k,k+1,\ell,\ell+1\}$.\\
We now use that
$$\max_{r\in\{k,k+1,\ell,\ell+1\}} \|\varphi_r^{n+1}(x)-x\|_p\le \frac{\sigma(-(n+1))}{\eta}$$
to get
\begin{align*}
\|(Q_{r_1}^{n+1}+\dots+Q_{r_s}^{n+1})(f(x)-f(y))\|_Y &\ge 2^{-(n+1)}\Big(\|x-y\|_p-\frac{2\sigma(-(n+1))}{\eta}\Big)\\
&\ge 2^{-(n+1)}\frac{\eta-2}{\eta} \|x-y\|_p.
\end{align*}
Since $\|Q_{r_1}^{n+1}+\cdots+Q_{r_s}^{n+1}\|\le 8(1+\gamma)(1+\varepsilon)(1+\beta)$, we have
$$\|f(x)-f(y)\|_Y\ge \frac{2^{-(n+1)}(\eta-2)\|x-y\|_p}{8\eta(1+\gamma)(1+\varepsilon)(1+\beta)}
\ge\frac{2^{\mu(\|x-y\|_p)}(\eta-2)\|x-y\|_p}{16\eta(1+\gamma)(1+\varepsilon)(1+\beta)}.$$

\medskip
Assume now that $\|x-y\|_p\ge \sigma(-1)$ or equivalently $-1\le \mu(\|x-y\|_p)$. Then using the maps $Q_k^{-1}$, $Q_{k+1}^{-1}$, $Q_{\ell}^{-1}$, and $Q_{\ell+1}^{-1}$ in place of $Q_k^{n+1}$, $Q_{k+1}^{n+1}$, $Q_{\ell}^{n+1}$, and $Q_{\ell+1}^{n+1}$, we obtain that
$$\|f(x)-f(y)\|_Y\ge\frac{(\eta-2)\|x-y\|_p}{16\eta(1+\gamma)(1+\varepsilon)(1+\beta)}.$$
Since $\mu\le 0$, it follows again that
$$\|f(x)-f(y)\|_Y\ge\frac{2^{\mu(\|x-y\|_p)}(\eta-2)\|x-y\|_p}{16\eta(1+\gamma)(1+\varepsilon)(1+\beta)}.$$
Since $\beta, \varepsilon, \gamma$ can be taken as small as wanted, and $\eta$ as large as needed, this finishes the proof of Theorem \ref{Lp0}.
\end{proof}

\section{Nearly isometric embeddability of stable metric spaces}\label{stablesection}
In this section we shall deal with a class of spaces strictly containing the class of proper metric spaces, namely the class of stable metric spaces. The notion of stability was introduced, originally for (separable) Banach spaces, in the work of Krivine and Maurey  \cite{KrivineMaurey1981}, in an attempt to exhibit a class of Banach spaces with a very regular linear structure. It seems that its natural extension to general metric spaces was first studied by Garling \cite{Garling1982}. A metric space $(X,d_X)$ is said to be stable if for any two bounded sequences $\xn$, $\yn$, and any two non-principal ultrafilters $\cU,\cV$ on $\bn$, the equality $\lim_{m,\cU}\lim_{n,\cV}d_{X}(x_m,y_n)=\lim_{n,\cV}\lim_{m,\cU}d_{X}(x_m,y_n),$ holds. Any $L_p$-space is actually stable for $p\in[1,\infty)$.

\medskip

In the remarkable article \cite{Kalton2007}, Kalton showed that for every $s\in(0,1)$, every stable metric space $M$ admits a strong embedding $f$ into some reflexive space, that depends heavily on $M$, such that $\omega_f(t)\le \max\{t,t^s\}$ and $\rho_f(t)\ge\min\{t,t^s\}$. There are actually good reasons why this embedding fails short to be an isometric  embedding. Indeed, a classical differentiability argument will tell you that the stable space $L_1$ does not even admit a bi-Lipschitz embedding into a Banach space with the Radon-Nikod\'ym property, and in particular into a reflexive space. For the same reason, $\ell_1$ does not admit a bi-Lipschitz embedding into a reflexive Banach space. Also note that Kalton's embeddings are not range embeddings. In this section we show that a slight modification of Kalton's construction gives embeddings with better properties. In particular we can produce range isometric embeddings, i.e. the constant $D$ and the scaling factor $r$ are equal to $1$ in the definition of a range bi-Lipschitz embedding. The sharpened embeddability result is better grasped after introducing the new notion of \textit{nearly isometric embeddability}.

\begin{defn}\label{nearlyisometric}
Let
$$\begin{array}{ll}
\cP:=\{\rho\colon\br^+\to\br^+|\ \rho \textrm{ is continuous, } \rho(t)=t \textrm{ for all }t\in [0,1],&\hskip -.3cm \rho(t)\le t \textrm{ for all }t\ge 1,\\
 & \ds\lim_{t\to +\infty}\frac{\rho(t)}{t}=0\},
\end{array}$$
and
$$\begin{array}{cl}
\Omega:=\{\omega\colon\br^+\to\br^+\ |\ \omega \textrm{ is continuous, }&\omega(0)=0,\ t\le \omega(t) \textrm{ for all }t\in [0,1],\\
& \ds\lim_{t\to 0}\frac{\omega(t)}{t}=\infty,\ \omega(t)=t\textrm{ for all }t\ge 1\}.
\end{array}$$
We say that $(X,d_X)$ nearly isometrically embeds into $(Y,d_Y)$ if there exist a family $(f_{\rho,\omega})_{(\rho,\omega)\in(\cP,\Omega)}$ of maps from $X$ into $Y$ such that $\rho(t)\le\rho_{f_{\rho,\omega}}(t)$ and $\omega_{f_{\rho,\omega}}(t)\le\omega(t)$. In other words, $(X,d_X)$ nearly isometrically embeds into $Y$ if for any pair of continuous functions $\rho, \omega:[0,+\infty)\to [0,+\infty)$ satisfying
\begin{enumerate}[(i)]
\item $t\le \omega(t)$ for $t\in [0,1]$ and $\omega(t)=t$ for $t\in[1,\infty)$,
\item $\omega(0)=0$ and $\lim_{t\to 0}\frac{\omega(t)}{t}=+\infty$,
\item $\rho(t)=t$ for $t\in [0,1]$ and $\rho(t)\le t$ for $t\in[1,\infty)$,
\item $\lim_{t\to +\infty}\frac{\rho(t)}{t}=0$,
\end{enumerate}
there exists a map $f\colon X\to Y$ such that for all $x,y\in X$
$$ \rho(d_X(x,y))\le d_Y(f(x),f(y)) \le \omega(d_X(x,y)).$$
\end{defn}

It is clear from Definition \ref{nearlyisometric} that if $X$ admits an isometric embedding into $Y$, then $X$ nearly isometrically embeds into $Y$.
%Almost Lipschitz embeddability and nearly isometric embeddability seem to be two incomparable notions. On the one hand, nearly isometric embeddability is not even implied by bi-Lipschitz embeddability. On the other hand, nearly isometric embeddability does not even imply bi-Lipschitz embeddability for large distances (resp. small distances) in regards of condition (iv) (resp. (ii)).
The notion of nearly isometric embeddability expresses the fact that one can construct an embedding that is as close as one wants to an isometric embedding. The following observation, whose easy proof is left to the reader, should justify the terminology.

\begin{prop}
If $X$ nearly isometrically embeds into $Y$ then $X$ admits an isometric range embedding with arbitrarily large range, i.e. for every $0<s_1\le s_2<\infty$ there exists $f\colon X\to Y$ such that $d_Y(f(x),f(y))=d_X(x,y)$, for all $x,y\in X$ satisfying $d_X(x,y)\in[s_1,s_2]$.
\end{prop}

As for the case of almost Lipschitz embeddability one can work with more regular maps.

\begin{lem}\label{technical} Let $\omega\in\Omega$ and $\rho\in \cP$, then
\begin{enumerate}[(i)]
\item there exists $\rho^*\in \cP$ such that $\rho^*\ge \rho$, $\rho^*$ is non-decreasing and the map $t\mapsto \frac{\rho^*(t)}{t}$ is non-increasing,\\
\item there exists $\omega^*\in\Omega$ such that $\omega^*\le \omega$, $\omega^*$ is non-decreasing and the map $t\mapsto \frac{\omega^*(t)}{t}$ is non-increasing.
\end{enumerate}
\end{lem}

\begin{proof} For $t>0$, we denote $\varepsilon(t)=\frac{\rho(t)}{t}$. To prove $(i)$ we need to find $\varepsilon^*$ non-increasing, such that $\varepsilon\le \varepsilon^*\le 1$, $\varepsilon^*=1$ on $[0,1]$, $t\mapsto t\varepsilon^*(t)$ is non-decreasing and $\lim_{t\to \infty}\varepsilon^*(t)=0$. First, we may replace $\varepsilon$ by $t\mapsto \sup_{s\ge t}\varepsilon(s)$ and thus assume that $\varepsilon$ is non-increasing. Then we set
$$\forall t\ge 0\ \ \varepsilon^*(t)=\frac{1}{t}\sup_{s\le t} s\varepsilon(s).$$
It is then clear that $t\mapsto t\varepsilon^*(t)$ is non-decreasing, $\varepsilon\le \varepsilon^*\le 1$ and $\varepsilon^*=1$ on $[0,1]$. It easy to check that $\lim_{t\to \infty}\varepsilon^*(t)=0$. Finally, we have to show that $\varepsilon^*$ is also non-increasing. So let $t<t'$. If $s\in [0,t]$, then $\frac{s\varepsilon(s)}{t'}\le \frac{t\varepsilon^*(t)}{t'}\le \varepsilon^*(t)$. On the other hand, if $s\in [t,t']$, we use the fact that $\varepsilon$ is non-increasing to get that $\frac{s\varepsilon(s)}{t'}\le\frac{s}{t'}\varepsilon(t)\le \varepsilon(t)\le \varepsilon^*(t)$. This finishes the proof of $(i)$.\\
The statement $(ii)$ can be deduced from $(i)$ by a change of variable.
\end{proof}

We will need the following weakening of Definition \ref{nearlyisometric}:

\begin{defn}
Let $\cC$ be a class of metric spaces. We say that $(X,d_X)$ nearly isometrically embeds into the class $\cC$ if there exist a family of spaces $(Y_{\gamma,\omega})_{(\rho,\omega)\in(\cP,\Omega)}$ in $\cC$, and a family $(f_{\rho,\omega})_{(\rho,\omega)\in(\cP,\Omega)}$ of maps from $X$ into $Y_{\rho,\omega}$ such that $\rho(t)\le\rho_{f_{\rho,\omega}}(t)$ and $\omega_{f_{\rho,\omega}}(t)\le\omega(t)$.
\end{defn}

The main result of this section is:
\begin{thm}\label{stablethm}
Let $(M,d)$ be a stable metric space. Let $(\rho,\omega)\in\cP\times\Omega$ such that $\rho,\omega$ are non-decreasing and $t\mapsto \frac{\rho(t)}{t}$, $t\mapsto \frac{\omega(t)}{t}$ are non-increasing. Then there exists a reflexive Banach space $(Y,\|\ \|)$ and a map $f\colon M\to Y$ such that for all $x,y \in M$, $$\rho(d(x,y))\le \|f(x)-f(y)\| \le \omega(d(x,y)).$$
\end{thm}

Hence, in view of Lemma \ref{technical}, we have:

\begin{cor}\label{stable}
Let $M$ be a stable metric space, then $M$ is nearly isometrically embeddable into the class of reflexive Banach spaces.
\end{cor}

The proof of Theorem \ref{stablethm} will fill in the rest of this section.
\begin{proof}[Proof of Theorem \ref{stablethm}] Fix an arbitrary point (denoted $0$) in $M$, and define
$$\text{Lip}_0^\omega(M):=\{g\colon M\to M\ |\ g(0)=0\ \text{and}\ \sup_{x\neq y}\frac{|g(y)-g(x)|}{\omega(d(x,y))}<+\infty\},$$
and
$$\forall g\in \text{Lip}_0^\omega(M),\ \ N_\omega(g)=\sup_{x\neq y}\frac{|g(y)-g(x)|}{\omega(d(x,y))}.$$
It is easily checked that $(\text{Lip}_0^\omega(M),N_\omega)$ is a Banach space. Define now the map $\delta\colon M\to \text{Lip}_0^\omega(M)^*$ by $\delta(x)(g)=g(x)$. By construction $N_\omega^*(\delta(x)-\delta(y))\le \omega(d(x,y))$, for all $x,y\in M$ where $N_\omega^*$ denotes the dual norm. Then, following \cite{Kalton2007}, for $(p,q) \in \tilde M=\{(p,q)\in M^2 \colon\ p\neq q\}$ and $x\in M$, we define $$g_{p,q}(x)=\max\{d(p,q)-d(q,x);0\}-\max\{d(p,q)-d(q,0);0\}.$$ Note that
\begin{align}
\forall x,y\in M\ \ |g_{p,q}(x)-g_{p,q}(y)|\le \min\{d(p,q);d(x,y)\}.
\end{align}
Now let
$$h_{p,q}=\frac{\rho(d(p,q))}{d(p,q)}g_{p,q}.$$
Note that $N_\omega(h_{p,q})\le 1$.

\medskip

We will need the following.
\begin{claim}\label{limits} For $(p,q) \in \tilde M$ and $(x,y)\in \tilde M$, consider $\ds R_{p,q}(x,y)=\frac{|h_{p,q}(x)-h_{p,q}(y)|}{\omega(d(x,y))},$ then
\begin{enumerate}[(i)]
\item $R_{p,q}(x,y)\to 0$ as $d(p,q)\to +\infty$ or $d(p,q)\to 0$, uniformly in $(x,y)\in \tilde M$,
\item $R_{p,q}(x,y)\to 0$ as $d(x,y)\to +\infty$ or $d(x,y)\to 0$, uniformly in $(p,q)\in \tilde M$.
\end{enumerate}
\end{claim}
\begin{proof}[Proof of Claim \ref{limits}]  We have that
$$ \forall (x,y)\in \tilde M \ \ R_{p,q}(x,y)\le \frac{\rho(d(p,q))}{d(p,q)}\frac{\min\{d(p,q);d(x,y)\}}{\omega(d(x,y))}.$$

$(i)$ a) For all $(x,y)\in \tilde M$, $\ds R_{p,q}(x,y)\le \frac{\rho(d(p,q))}{d(p,q)}\frac{d(x,y)}{\omega(d(x,y))}\le \frac{\rho(d(p,q))}{d(p,q)}$, and $R_{p,q}(x,y)\to 0$ as $d(p,q)\to +\infty$  uniformly in $(x,y)\in \tilde M$.

\smallskip

b) Assume $d(p,q)\le 1$.\\
If $d(x,y)\le d(p,q)$ then $\ds R_{p,q}(x,y)\le\frac{d(x,y)}{\omega(d(x,y))}\le \frac{d(p,q)}{\omega(d(p,q))},$ because $\ds t\mapsto \frac{t}{\omega(t)}$ is non-decreasing. Otherwise
$$R_{p,q}(x,y)\le\frac{d(p,q)}{\omega(d(x,y))}
=\frac{\omega(d(p,q))}{\omega(d(x,y))}\frac{d(p,q)}{\omega(d(p,q))}
\le\frac{d(p,q)}{\omega(d(p,q))},$$
because $\omega$ is non-decreasing.\\
Therefore $R_{p,q}(x,y)\to 0$ as $d(p,q)\to 0$,  uniformly in $(x,y)\in \tilde M$.

\smallskip

$(ii)$ a) Assume that $d(x,y)>1$.\\

If $d(p,q)\le d(x,y)$ then $\ds R_{p,q}(x,y)\le\frac{\rho(d(p,q))}{d(p,q)}\frac{d(p,q)}{d(x,y)}\le \frac{\rho(d(x,y))}{d(x,y)},$
because $\rho$ is non-decreasing.\\
If $d(p,q)\ge d(x,y)$ then $\ds R_{p,q}(x,y)\le\frac{\rho(d(p,q))}{d(p,q)}\le \frac{\rho(d(x,y))}{d(x,y)},$
because $\ds t\to\frac{\rho(t)}{t}$ is non-increasing.\\
Therefore $R_{p,q}(x,y)\to 0$ as $d(x,y)\to +\infty$  uniformly in $(p,q)\in \tilde M$.

\smallskip

b) If $d(x,y)\le 1$, then
$$R_{p,q}(x,y)\le\frac{\rho(d(p,q))}{d(p,q)}\frac{d(x,y)}{\omega(d(x,y))}\le \frac{d(x,y)}{\omega(d(x,y))}.$$
Therefore $R_{p,q}(x,y)\to 0$ as $d(x,y)\to 0$  uniformly in $(p,q)\in \tilde M$.\\
\end{proof}

Then we deduce:

\begin{claim}\label{weakcompact} The set $W=\{h_{p,q};\ (p,q)\in \tilde M\}$ is weakly relatively compact in $\text{Lip}_0^\omega(M)$.
\end{claim}
\begin{proof}[Proof of Claim \ref{weakcompact}] Here one just has to mimic the corresponding part of the proof of Theorem 2.1 in \cite{Kalton2007}. Let us recall the details. It follows from the Eberlein-\v{S}mulian theorem that it is enough to show that any sequence $(h_n)_{n=1}^\infty=(h_{p_n,q_n})_{n=1}^\infty$, with $(p_n,q_n)\in \tilde M$ admits a weakly convergent subsequence. The closed linear span of $\{h_n\colon n\ge 1\}$, denoted $[h_n\colon n\ge 1]$ is separable. So, there exists a countable subset $M_0$ of $M$ containing $0$ and all $p_n,q_n$ for $n\ge 1$ such that
$$\forall g\in [h_n\colon n\ge 1]\ \ \|g\|_{\text{Lip}_0^\omega(M)}=\|g\|_{\text{Lip}_0^\omega(M_0)}.$$
Using a diagonal argument, we may assume that  $(h_n)_{n=1}^\infty$ converges pointwise on $M_0$ to a function $h$. We may also assume that $\lim_{n\to +\infty}d(p_n,q_n)=r\in [0,+\infty]$ and that for all $x\in M_0$ $\lim_{n\to +\infty}d(x,q_n)$ exists in $[0,+\infty]$. In particular, denote $s=\lim_{n\to +\infty}d(0,q_n)$. Now we define $V\colon [h_n\colon n\ge 1]\to \ell_\infty(\tilde{M_0})$, where $\tilde{M_0}=\{(p,q)\in M_0^2\colon p\neq q\}$ by
$$\forall g\in \text{Lip}_0^\omega(M),\ \ Vg=\Big(\frac{g(x)-g(y)}{\omega(d(x,y))}\Big)_{(x,y)\in \tilde{M_0}}.$$
By construction, $V$ is a linear isometry from $[h_n\colon n\ge 1]$ into $\ell_\infty(\tilde{M_0})$. So it is enough for us to show that the sequence $(Vh_n)_{n=1}^\infty$ is weakly convergent to $Vh$. Following \cite{Kalton2007}, we denote $A$ the closed subalgebra of $\ell_\infty(\tilde{M_0})$, generated by the constant functions, $V([h_n;\ n\ge 1])$, the maps $(x,y)\mapsto \arctan(d(x,u))$ and $(x,y)\mapsto \arctan(d(y,u))$, for $u\in M_0$ and the map $(x,y)\mapsto \arctan(d(x,y))$. Since $A$ is separable, there exists a metrizable compactification $K$ of $\tilde{M_0}$ such that every $f$ in $A$ admits a unique continuous extension to $K$. With this description of $A$, it follows from the dominated convergence theorem that we only need to prove that for all $\xi\in K$, $\lim_{n\to\infty}Vh_n(\xi)=Vh(\xi)$.\\
If $r=0$ or $r=+\infty$, it follows from $(i)$ in Claim \ref{limits} that $\lim_{n\to \infty}\|Vh_n\|=0$. Therefore $\lim_{n\to \infty}\|h_n\|=0$ and $h\equiv0$. Thus, we may assume that $0<r<+\infty$.\\
Let now $\xi\in K$. Then pick $(x_m,y_m)\in \tilde{M_0}$ such that $(x_m,y_m)\to \xi$ in $K$ and denote $t=\lim_{m\to \infty}d(x_m,y_m)$. \\
First note that if $t=0$ or $t=+\infty$, it follows from $(ii)$ in Claim \ref{limits} that for all $n\ge 1$, $Vh_n(\xi)=Vh(\xi)=0$.\\
Thus, for the sequel, we may and will assume that $0<r<+\infty$ and $0<t<+\infty$.
By taking a further subsequence, we may also assume that $\lim_m\lim_n d(q_n,x_m)$ exists in $[0,+\infty]$. Then it follows from the stability of $M$ that $\lim_m\lim_n d(q_n,x_m)=\lim_n\lim_m d(q_n,x_m)$ and therefore that $\lim_m\lim_n h_n(x_m)=\lim_n\lim_m h_n(x_m)$.\\
Since $Vh\in C(K)$ and $(h_n)_{n\ge 1}$ converges pointwise to $h$ on $M_0$, we have
$$Vh(\xi)=\lim_{m\to \infty}Vh(x_m,y_m)=\lim_{m\to \infty}\frac{h(x_m)-h(y_m)}{\omega(d(x_m,y_m))}=\lim_{m\to\infty}\lim_{n\to \infty}\frac{h_n(x_m)-h_n(y_m)}{\omega(d(x_m,y_m))}.$$
Finally, since $M$ is stable  and $Vh_n \in C(K)$ we obtain
$$Vh(\xi)=\lim_{n\to\infty}\lim_{m\to \infty}\frac{h_n(x_m)-h_n(y_m)}{\omega(d(x_m,y_m))}=
\lim_{n\to\infty}Vh_n(\xi).$$
This concludes the proof of this Claim.

%If $L$ is the spectrum of $B=\ell_\infty(\tilde{M_0})$, it is easily seen that $\{\xi\restrict{A},\ \xi\in L\}$ is closed in $(K,\sigma(A,A^*))$. It also norming for $A=C(K)$. Then it follows that $K=\{\xi\restrict{A},\ \xi\in L\}$. Now, it known that $L$ is the Stone-Cech compactification of $\tilde{M_0}$.

\end{proof}

We now proceed with the construction of the embedding $f$. Consider the operator\\ $S:\ell_1(W)\to (\text{Lip}_0^\omega(M),N_\omega)$ defined by
$$\forall \xi=(\xi_h)_{h\in W}\in \ell_1(W)\ \ S(\xi)=\sum_{h\in W}\xi_hh.$$
Since every $h \in W$ is in the unit ball of $N_\omega$, we have that $\|S\|\le 1$. Moreover, it follows from Claim \ref{weakcompact} that $S$ is a weakly compact operator. Then the isometric version \cite{LimaLimaOja2000} of a factorization theorem of Davis, Figiel, Johnson, and Pe\l czy\'nski \cite{DavisFJP74} yields the existence of a reflexive Banach space $X$, and of linear maps $T:\ell_1(W)\to X$ and $U:X\to  \text{Lip}_0^\omega(M)$ such that $\|T\|\le 1$, $\|U\|\le 1$, and $S=UT$. Then we define $f\colon M\to Y=X^*$ by
$$\forall x\in M\ \ f(x)=U^*(\delta(x)).$$
First, we clearly have that
$$\forall x,y\in M\ \ \|f(x)-f(y)\|\le \|U^*\|\,N_\omega^*\big(\delta(x)-\delta(y)\big)\le \omega(d(x,y)).$$
On the other hand, since $S^*=T^*U^*$ and $\|T^*\|\le 1$, we have that for all $(x,y)\in \tilde M$:
$$\|f(x)-f(y)\|\ge \|S^*(\delta(x)-\delta(y))\|_{\ell_\infty(W)}\ge\big\langle h_{x,y},\delta(x)-\delta(y)\big\rangle=\rho(d(x,y)).$$
This finisshes the proof, as $f$ provides the desired embedding into the reflexive Banach space $Y=X^*$.

\end{proof}

\section{Optimality}\label{optimality}

 Our first statement describes what can be said of a Banach space that contains a bi-Lipschitz copy of every compact subset of a given separable Banach space. This is a generalization of an unpublished argument due to N. Kalton in the particular case of $X=c_0$, that was already mentioned in \cite{Baudierthesis}.

\begin{prop}\label{kalton} Let $X$ be a separable Banach space.

\begin{enumerate}[(i)]
\item There exists a compact subset $K$ of $X$ such that, whenever $K$ bi-Lipschitzly embeds into a Banach space $Y$, then $X$ linearly embeds into $Y^{**}$. In particular, $X$ is crudely finitely representable into $Y$. If moreover, $Y$ has the Radon-Nikod\'{y}m property, then $X$ linearly embeds into $Y$.\smallskip

\item For any $p\in [1,\infty)$, there exits a compact subset $K_p$ of $L_p$ which is almost Lipschitzly embeddable but not bi-Lipschitzly embeddable in $\ell_p$ or $\Big(\sum_{n=1}^\infty \ell_p^n\Big)_{\ell_2}$.\smallskip

\item There exists a compact subset $K_\infty$ of $c_0$ which is almost Lipschitzly embeddable but not bi-Lipschitzly embeddable in $\Big(\sum_{n=1}^\infty \ell_\infty^n\Big)_{\ell_2}$.
\end{enumerate}
\end{prop}

\begin{proof} Since $(ii)$ and $(iii)$ are easy consequences of $(i)$, Corollary \ref{Lp} and classical linear Banach space theory, we just prove $(i)$. Since $X$ is separable, we can find a biorthogonal system  $(x_n,x_n^*)_{n=1}^\infty$ in $X\times X^*$ such that the linear span of $\{x_n\colon\ n\ge 1\}$ is dense in $X$ (see \cite{Markushevich1943}). We now pick a decreasing sequence $(a_n)_{n=1}^\infty$ of positive real numbers such that
$$\sum_{n=1}^\infty a_n\|x_n\|\,\|x_n^*\|\le 1,$$
and for all $x\in X$ we define
$$S(x)=\sum_{n=1}^\infty a_n x_n^*(x)x_n.$$
Then $S$ is clearly a compact operator on $X$ and $\|S\|\le 1$. Finally, the norm closure of $S(B_X)$ is a compact subset of $X$, denoted by $K$ in the sequel.

\noindent Assume now that $f:K\to Y$ is a map such that for all $x,x'\in K$,
$$\|x-x'\|_X\le \|f(x)-f(x')\|_Y\le C\|x-x'\|_X.$$ Then the map $f\circ S$ is $C$-Lipschitz on the open unit ball $B_X$ of $X$. Let us consider $f\circ S$ as a map from $B_X$ to $Y^{**}$. Since $X$ is separable, we can actually consider $f$ as map from $B_X$ to the dual $E^*$ of a separable Banach space $E$ (see Proposition F.8 and the proof of Corollary 7.10 in \cite{BenyaminiLindenstrauss2000}). Then it follows from the work of Heinrich and Mankiewicz \cite{HeinrichMankiewicz1982} (see also Corollary 6.44 in \cite{BenyaminiLindenstrauss2000}), that $B_X\setminus W$ is Gauss-null, where $W$ is the set of all $x\in B_X$ such that $f\circ S$ is weak$^*$-Gâteaux differentiable at $x$. We again refer to \cite{BenyaminiLindenstrauss2000} for the definition and properties of Gauss-null sets. For $x\in W$, we denote $D^*_{f\circ S}(x)$ the weak$^*$-Gâteaux derivative of $f\circ S$ at $x$. We will need the following.

\begin{claim}\label{Gaussnull} For any $k$ in the linear span of $\{x_n\colon\ n\ge 1\}$ and any  $\delta<1$, the set
$$W_{k,\delta}=\{x\in W\colon\ \|D^*_{f\circ S}(x)(k)\|< \delta\|S(k)\|\}$$
is Gauss-null
\end{claim}

\begin{proof}[Proof of Claim \ref{Gaussnull}] We will mimic the proof of Theorem 7.9 in \cite{BenyaminiLindenstrauss2000} (also due to Heinrich and Mankiewicz \cite{HeinrichMankiewicz1982}). Let us assume as we may that $\|k\|=1$. To see that $W_{k,\delta}$ is Gauss-null, it is enough to prove that for any line $L$ in the direction of $k$, $L\cap W_{k,\delta}$ is of Lebesgue measure (denoted $m$) equal to 0. If not, there exists a density point $z_0\in L\cap W_{k,\delta}$. Then we can find $\eps>0$ such that $ m([z_0,z_o+\eps k]\cap W_{k,\delta})> \eps(1-\frac{1-\delta}{C})$.\\
We denote $A=[z_0,z_o+\eps k]\cap W_{k,\delta}$ and $B= [z_0,z_o+\eps k]\setminus A$.\\
Since $S$ is linear, we have $\|(f\circ S)(z_0+\eps k)-(f\circ S)(z_0)\|\ge \eps\|S(k)\|$. Thus, there exists $y^*$ in the unit sphere of $Y^*$ such that $$\big\langle (f\circ S)(z_0+\eps k)-(f\circ S)(z_0),y^*\big\rangle\ge \eps\|S(k)\|.$$ Consider now $\varphi(t)=\big\langle (f\circ S)(z_0+t k)-(f\circ S)(z_0),y^*\big\rangle$. Using again the linearity of $S$,we note that $\varphi$ is $C\|S(k)\|$-Lipschitz. Thus $\varphi$ is almost everywhere differentiable and

\begin{align*}
              &\eps\|S(k)\| \le |\varphi(\eps)-\varphi(0)|=\int_0^\eps \varphi'(t)\, dt \le \int_A \delta\|S(k)\|\,dm+\int_BC\|S(k)\|\, dm\\
              &\le \eps \delta \|S(k)\|+ m(B)C\|S(k)\|< \eps \delta \|S(k)\|+ \frac{\eps(1-\delta)}{C}C\|S(k)\|= \eps\|S(k)\|,
\end{align*}
which is a contradiction.

\end{proof}

We now proceed with the proof of Proposition \ref{kalton} and pick $D$ a countable dense subset of the linear span of $\{x_n\colon\ n\ge 1\}$. Then, it follows from Claim \ref{Gaussnull} that there exists $x\in W$ such that for all $k\in D$,  and therefore for all $k\in X$, $$\|D^*_{f\circ S}(x)(k)\|\ge \|S(k)\|.$$
On the other hand, it is clear that for all $k\in X$ $$\|D^*_{f\circ S}(x)(k)\|\le C\|S(k)\|.$$

Let now $h$ be in the linear span of $\{x_n\colon\ n\ge 1\}$. There exists $N\in \mathbb N$ such that $h=\sum_{n=1}^N x_n^*(h)x_n$. Then, we have that $h=S(k)$, where $\ds k=\sum_{n=1}^N \frac{1}{a_n} x_n^*(h)x_n$ belongs to the linear span of $\{x_n\colon\ n\ge 1\}$. Define $V(h):=D^*_{f\circ S}(x)(k)$. It follows from the above inequalities that $\|h\|\le \|V(h)\|\le C\|h\|$. Therefore $V$ extends to a linear embedding from $X$ into $Y^{**}$.

\smallskip

If $Y$ has the Radon-Nikod\'{y}m property, then $f\circ S$ admits a point of Gâteaux-differentiability. It then follows, with similar but easier arguments, that $X$ linearly embeds into $Y$.
\end{proof}

Using a classical argument of G. Schechtman, we deduce that Corollary \ref{Lp} is optimal.

\begin{cor}\label{optimal} Let $X$ be a separable Banach space.

\begin{enumerate}[(i)]
\item There exists a compact subset $K$ of $X$ such that, whenever $K$ almost Lipschitzly embeds into a Banach space $Y$, then $X$ is crudely finitely representable into $Y$.\smallskip
\item There exists a compact subset $K$ of $X$ such that, whenever $K$ nearly isometrically embeds into a Banach space $Y$, then $X$ is finitely representable into $Y$.\smallskip

\item For any $p\in [1,\infty)$, there exists a compact subset $K_p$ of $L_p$ such that,
$K_p$ almost Lipschitzly embeds into $Y$ if and only if $Y$ uniformly contains the $\ell_p^n$'s.\smallskip

\item There exists a compact subset $K_\infty$ of $c_0$ such that,
$K_\infty$ almost Lipschitzly embeds into $Y$ if and only if $Y$ uniformly contains the $\ell_\infty^n$'s.
\end{enumerate}
\end{cor}

\begin{proof} Note that $(iii)$ and $(iv)$ are easy consequences of $(i)$ together with Corollary \ref{Lp}. We prove $(i)$. Consider $K$ the compact given by Proposition \ref{kalton}. Since $K$ is compact, we can build am increasing sequence $(R_n)_{n=1}^\infty$ of finite subsets of $K$ such that for any $n\ge 1$, $R_n$ is a $2^{-n}$-net of $K$. Assume that $K$ is almost Lipschitz embeddable into a Banach space $Y$, then there exists $D\in[1,\infty)$ such that for every $n\ge 1$ and $s_n=\inf\{\|x-y\|_X\colon\ x\neq y \in R_n\}>0$, $K$ admits a range embedding with range $[s_n,\diam(K)]$ and distorsion bounded by $D$. In other words, there exists $D\in[1,\infty)$ such that for any $n\ge 1$ there exists a map $f_n: K\to Y$ so that
\begin{equation}
\forall x,y\in R_n\ \ \|x-y\|_X\le \|f_n(x)-f_n(y)\|_Y\le D\|x-y\|_X.
\end{equation}
We can also define $c_n:K\to R_n$ such that for any $x\in K$, $\|x-c_n(x)\|_X=d(x,R_n)\le 2^{-n}$. Then we set $g_n=f_n\circ c_n$. Finally, consider a non-principal ultrafilter $\mathcal U$ on $\mathbb N$ and define $f$ from $K$ to the ultrapower $Y_{\mathcal U}$ of $Y$ by $f(x)=(g_n(x))_{\mathcal U}$, for $x\in K$.
It is easy to check that $f$ is a bi-Lipschitz embedding of $K$ into $Y_{\mathcal U}$. Then it follows from Proposition \ref{kalton} that $X$ is isomorphic to a subspace of $Y_{\mathcal U}^{**}$. Finally, the conclusion follows from the local reflexivity principle and the finite representability of $Y_{\mathcal U}$ into $Y$.

\smallskip

\noindent For $(ii)$ just remark that $D$ can be taken to be $1$ in the argument above and the proof of Proposition \ref{kalton} shows that $X$ is isometric to a subspace of $Y_{\mathcal U}^{**}$ which is finitely representable in $Y$.
\end{proof}

M. Ostrovskii proved in \cite{Ostrovskii2012} a nice result about the finite determinacy of bi-Lipschitz embeddability for locally finite spaces. The heart of the proof is that if a Banach space $X$ is crudely finitely representable into another infinite-dimensional Banach space $Y$, then there exists a constant $D\in[1,\infty)$ such that any locally finite subset of $X$ is bi-Lipschitzly embeddable into $Y$ with distorsion at most $D$. It is an easy consequence of the proof of item $(i)$ in Corollary \ref{optimal} that for two infinite-dimensional Banach spaces $X$ and $Y$, if there exists $D\in[1,\infty)$ such that any finite subset of $X$ is bi-Lipschitzly embeddable into $Y$ with distortion st most $D$, then $X$ is crudely finitely representable in $Y$. So we can state.

\begin{cor} Let $X$ and $Y$ be two infinite-dimensional Banach spaces.
\smallskip

\noindent $X$ is crudely finitely representable in $Y$ if and only if there exists $D\in[1,\infty)$ such that any locally finite subset of $X$ is bi-Lipschitzly embeddable into $Y$ with distortion at most $D$.
\end{cor}

\section{Concluding Remarks}
As we will see shortly the new notions of embeddability discussed in this article are incomparable in full generality. Nevertheless understanding the relationship between these various notions of embeddability, in particular in the Banach space framework, seems to be a delicate issue.  In this section we want to highlight a few observations and raise some related questions.

\smallskip

As shown in Proposition \ref{kalton} the notion of almost Lipschitz embeddability is strictly weaker than bi-Lipschitz embeddability since there are examples of compact metric spaces that do almost Lipschitz embed into a Banach space without being bi-Lipschitzly embeddable in it. In the Banach space setting the situation is more elusive. It follows from a result of Ribe \cite{Ribe1976} that $\ell_1$ admits a strong embedding (and also for every $s\in(0,\infty)$, a $[s,\infty)$-range embedding) into $\xbgq$ whenever $q\in(1,\infty)$ and $(p_n)_{n=1}^\infty\subset(1,\infty)$ with $\lim_{n\to\infty}p_n=1$. A classical differentiability argument rules out the existence of a bi-Lipschitz embedding of $\ell_1$ into the reflexive space $\xbgq$.

\begin{problem}
Let $(p_n)_{n=1}^\infty\subset(1,\infty)$ such that $\lim_{n\to\infty}p_n=1$. Is $\ell_1$ almost Lipschitz embeddable into $\xbgq$ for some $q\in(1,\infty)$?
\end{problem}

\begin{problem} Does there exist Banach spaces $X$ and $Y$ such that $X$ almost Lipschitz embeds into $Y$, but $X$ does not bi-Lipschitz embed into $Y$?
\end{problem}

It is worth noting that some elements of the proof of Corollary \ref{optimal} can be slightly generalized. Consider a non-principal ultrafilter $\cU$ on $\bn$. Fix an origin $t_0\in Y$, and define the ultrapower of $Y$ pointed at $t_0$ as $Y_\cU:=\ell_{\infty}(Y,d_Y,t_0)_{\slash\sim}$ where $\ell_{\infty}(Y,d_Y,t_0):=\{(x_n)_{n=1}^\infty\subset Y\colon \sup_{n\ge 1}d_Y(x_n,t_0)<\infty\}$, and $(x_n)_{n=1}^\infty\sim(y_n)_{n=1}^\infty$ if and only if $\lim_{n,\cU}d_Y(x_n,y_n)=0$.  $Y_\cU$ is a metric space and the distance is given by $d_{Y_\cU}(x,y)=\lim_{n,\cU}d_Y(x_n,y_n)$.

\begin{prop}\label{ultra} Let $X$ and $Y$ be metric spaces.

\begin{enumerate}[(i)]
\item If $X$ is almost Lipschitzly embeddable into $Y$, then $X$ bi-Lipschitzly embeds into any ultrapower of $Y$.\smallskip

\item If $X$ is nearly isometrically embeddable into $Y$, then $X$ isometrically embeds into any ultrapower of $Y$.\smallskip
\end{enumerate}
\end{prop}

\begin{proof} $(i)$ Assume that $X$ is almost Lipschitzly embeddable into $Y$, then there exist $r\in(0,\infty)$ and $D\in[1,\infty)$ such that for every $n\ge 1$, $X$ admits a range embedding into $Y$ with range $[\frac{1}{n},\infty)$. In other words, there exist $r\in(0,\infty)$ and $D\in[1,\infty)$ such that for any $n\ge 1$ there exists a map $f_n: X\to Y$ so that for all $x,y\in X$ with $d_X(x,y)\in[\frac{1}{n},\infty)$ one has
\begin{equation*}
rd_X(x,y)\le d_Y(f_n(x),f_n(y))\le Drd_X(x,y).
\end{equation*}
Define $f$ from $X$ to the ultrapower $Y_{\mathcal U}$ of $Y$ by $f(x)=(0,\dots, 0, f_n(x), f_{n+1}(x),\dots)_{\mathcal U}$, for $x\in X$ such that $d(x,t_0)\in[\frac{1}{n},\infty)$, and $f(t_0)=0$. It is easy to check that $f$ is a bi-Lipschitz embedding of $X$ into $Y_{\mathcal U}$.

\smallskip

\noindent $(ii)$ Assume that $X$ is nearly isometrically embeddable into $Y$, then for every $n\ge 1$ $X$ admits an isometric range embedding into $Y$ with range $[\frac{1}{n},n]$. In other words, for any $n\ge 1$ there exists a map $f_n: X\to Y$ so that for all $x,y\in X$ with $d_X(x,y)\in[\frac{1}{n},n]$ one has
\begin{equation}
d_Y(f_n(x),f_n(y))=d_X(x,y).
\end{equation}
Finally, consider a non-principal ultrafilter $\mathcal U$ on $\mathbb N$ and define $f$ from $X$ to the ultrapower $Y_{\mathcal U}$ of $Y$ by $f(x)=(0,\dots, 0, f_n(x), f_{n+1}(x),\dots)_{\mathcal U}$, for $x\in X$ such that $d(x,t_0)\in[\frac{1}{n},n]$, and $f(t_0)=0$.
It is easy to check that $f$ is an isometric embedding of $X$ into $Y_{\mathcal U}$.
\end{proof}
As mentioned previously almost Lipschitz embeddability or nearly isometric embeddability implies strong embeddability. It follows from Proposition \ref{ultra} that the converse does not hold. Indeed, it is well-known that $\ell_1$ admits a strong embedding into $\ell_p$, but $\ell_1$ does not bi-Lipschitzly embed into $L_p$ for any $p\in(1,\infty)$, and an ultrapower of an $L_p$-space is an $L_p$-space.

%It also follows from an easy modification of Proposition \ref{ultra} and Theorem \ref{stable} that every stable metric space embeds isometrically into an ultraproduct of reflexive spaces.

\smallskip

Recall that a well-known result of Schoenberg states that isometric embeddability into a Hilbert space is finitely determined (cf. \cite{WellsWilliams1975}). Note also that Godefroy and Kalton proved that a separable Banach space which is isometrically embeddable into a Banach space is necessarily linearly isometrically embeddable \cite{GodefroyKalton2003}. Besides, being Hilbertian is determined by the 2-dimensional subspaces. The following observations easily follow.

\begin{prop}
\begin{enumerate}[(i)]
\item If a metric space $X$ is nearly isometrically embeddable into a Hilbert space $H$ then $X$ is isometrically embeddable into $H$.
\item If a Banach space $X$ is nearly isometrically embeddable into a Hilbert space $H$ then $X$ is linearly isometrically embeddable into $H$.
\end{enumerate}
\end{prop}
An interesting consequence of observation $(ii)$ above is that nearly isometric embeddability is not implied by bi-Lipschitz embeddability. It is indeed easy to find a non Hilbertian equivalent renorming of $\ell_2$.

\smallskip

Let  $\alpha_Y(X)$ be the \textit{compression exponent} of $X$ in $Y$ ($Y$-compression of $X$ in short) introduced by Guentner and Kaminker \cite{GuentnerKaminker2004}. In other words, $\alpha_Y(X)$ is the supremum of all numbers $0\le \alpha\le 1$ so that there exist $f\colon X\to Y$, $\tau\in(0,\infty)$, and $C:=C(\tau)\in[1,\infty)$ such that if $d_{X}(x,y)\in[\tau,\infty)$ then
$$\frac{1}{C}d_X(x,y)^{\alpha}\le d_Y(f(x),f(y))\le C d_X(x,y).$$
The fact that for every $\alpha\in(0,1)$ the functions $(t\mapsto\min\{t,t^\alpha\})\in\cP$ and $(t\mapsto\max\{t,t^\alpha\})\in \Omega$, implies that if $X$ nearly isometrically embeds into $Y$ then $\alpha_Y(X)=1$. It follows from Kalton and Randrianarivony \cite{KaltonRandrianarivony2008} that $\ell_2$ is not almost Lipschitz embeddable into $\ell_p$ when $p\in[1,\infty)$ and $p\neq 2$. However it was shown by the first author \cite{BaudierJTA} that $\alpha_{\ell_p}(\ell_2)=1$ for $p\in[1,2)$.

\begin{problem}
Is $\ell_2$ nearly isometrically embeddable into $\ell_p$ for some $p\in[1,2)$?
\end{problem}

\begin{problem} Exhibit two metric spaces $X$ and $Y$ such that $X$ nearly isometrically embeds into $Y$, but $X$ does not almost Lipschitz embed into $Y$?
\end{problem}
%It follows from \cite{KaltonMA12} that $\co$ admits a bi-Lipschitz for large distances embedding into $\cZ(\co)$, but not bi-Lipschitz embedding.

%\begin{problem} Is $\co$ almost Lipschitz embeddable into $\cZ(\co)$?\end{problem}

We wish to conclude this article with an example of a Banach space containing $c_0$ in its bidual, which is interesting in view of Proposition \ref{kalton}. This example has been suggested to us by \'{E}. Ricard. It follows from Proposition \ref{kalton} that a Banach space which is universal for compact metric spaces and bi-Lipschitz embeddings contains a linear copy of $c_0$ in its bidual. On the other hand, a Banach space containing $c_0$ in its bidual is not necessarily universal for the class of compact metric spaces and bi-Lipschitz embeddings. Consider $X=\left(\sum_{n=1}^\infty \ell_{\infty}^n\right)_{\ell_1}=\left(\sum_{n=1}^\infty \ell_1^n\right)_{c_0}^*$.
$X$ contains a basic sequence equivalent to the canonical basis of $c_0$ in its bidual. Indeed, this sequence is given for $i\ge 1$ by $f_i=$ $w*$-$\displaystyle\lim_{n\to\infty}e_{n,i}\in X^{**}$, where $(e_{n,i})_{i=1}^n$ is the canonical basis of $\ell_{\infty}^n$. However, $X$ being a separable dual space it has the Radon-Nikod\'{y}m property and it does not contain a linear copy of $c_0$. By Proposition \ref{kalton}, $X$ cannot be universal for the class of compact metric spaces and bi-Lipschitz embeddings.

\bigskip
\noindent{\bf Acknowledgements.} The authors wish to thank to thank A. Proch\'{a}zka, \'{E}. Ricard and G. Schechtman for very useful discussions on the subject of this article. The first author would like also to warmly thank G. Godefroy for his support and guidance while he was in residence at the Institut de Math\'ematiques Jussieu-Paris Rive Gauche.

\end{document}